\documentclass[a4paper,12pt]{article}

\usepackage{amssymb}
\usepackage{epsfig}
\usepackage{amsfonts}
\usepackage{amsmath}
\usepackage{euscript}
\usepackage{amscd}
\usepackage{amsthm}
\usepackage{theoremref}
\usepackage{enumerate}
\usepackage{array}
\usepackage{mathrsfs}
\usepackage{tikz-cd}
\DeclareMathAlphabet{\mathpzc}{OT1}{pzc}{m}{it}
\DeclareMathAlphabet{\mathpzc}{OT1}{pzc}{m}{it}

\newtheorem{thm}{Theorem}[section]
\newtheorem{lem}[thm]{Lemma}
\newtheorem{prop}[thm]{Proposition} 
\newtheorem{cor}[thm]{Corollary}

\newtheorem{rem}[thm]{Remark}

\newcommand{\bZ}{\mathbb Z}
\newcommand{\Z}{\mathbb Z}
\newcommand{\bQ}{\mathbb Q}

\newcommand{\bC}{\mathbb C}

\newcommand{\A}{\mathbb A}

\newcommand{\X}{X_1,\dots,X_m}
\newcommand{\mi}{1\leqslant i \leqslant m}

\newcommand{\td}{\operatorname{tr.deg}}
\newcommand{\End}{\operatorname{End}}
\newcommand{\Aut}{\operatorname{Aut}}

\newcommand{\gr}{\operatorname{gr}}
\newcommand{\dk}{\operatorname{DK}}
\newcommand{\ml}{\operatorname{ML}}

\title{ A new family of counterexamples to the Zariski Cancellation Problem in positive characteristic
	
	\author{Parnashree Ghosh$^{*}$ and
		Ananya Pal$^{**}$\\
		{\small{\it  Theoretical Statistics and Mathematics  Unit, Indian Statistical Institute,}}\\
		{\small{\it 203 B.T.Road, Kolkata-700108, India}}\\
		{\small{\it e-mail :
				$^{*}$parnashree$\_$r@isical.ac.in, ghoshparnashree@gmail.com}}\\
		{\small {\it e-mail : $^{**}$palananya1995@gmail.com }}
	}
}

\begin{document}
	\date{}
	\maketitle
		\abstract{In this paper, over a field of positive characteristic we exhibit an infinite family of counter examples to the Zariski Cancellation Problem in higher dimensions ($\geqslant 3$) which are pairwise non-isomorphic and also non-isomorphic to the existing family of counter examples, demonstrated by Gupta in \cite{adv}.
	}
	\smallskip

\smallskip


	\section{Introduction}
	Throughout this paper $k$ will denote a field. For an integral domain $R$, $R^{[n]}$ denotes a polynomial ring in $n$ variables over $R$. We begin with one of the fundamental problems in the area of Affine Algebraic Geometry, namely the Zariski Cancellation Problem (ZCP).
	
	\medskip
	\noindent
	{\bf Question 1}: For an integer $n \geqslant 1$, is $k^{[n]}$ cancellative? That is if $B$ be an affine domain such that $B^{[1]}=k^{[n+1]}$, then does this imply that $B=k^{[n]}$? 
	
	\smallskip
	The answer to the above question is affirmative for $n=1$ (cf. \cite{aeh}). For $n=2$, when $k$ is a field of characteristic zero, the cancellative property of $k^{[2]}$ is proved by Miyanishi, Sugie and Fujita (\cite{miyasug}, \cite{fuj}). 
	This result is extended over perfect fields by Russell in \cite{russ}. 
	However, in \cite{bhgu}, Bhatwadekar and Gupta have dropped the condition of perfect fields and extended the result over arbitrary fields.
	Thus proving the cancellative properly of $k^{[2]}$, for any field.

	We now discuss the case for $n \geqslant 3$. When $k$ is a field of characteristic zero, the problem is still open. However, potential candidate for counterexample to this problem in characteristic zero is given by the Russell-Koras threefold whose coordinate ring is given by $k[X,Y,Z,T]/(X^2Y-X-Z^2-T^3)$. However, over fields of positive characteristic, the work of Gupta in \cite{inv} and \cite{adv} produced counterexamples ZCP for all $n \geqslant 3$. She has considered varieties whose coordinate rings are given by 
	\begin{equation}\label{A1}
		B= \dfrac{k[X_1,\ldots,X_m,Y,Z,T]}{(X_1^{r_1}\cdots X_m^{r_m}Y-f(Z,T))},
	\end{equation}
	where $k[Z,T]/(f)=k^{[1]}$, i.e., $f$ is a line in $k[Z,T]$. For $m=1$, such rings were introduced by Asanuma as an illustration of non-trivial $\A^2$-fibration over a DVR not containing $\bQ$. 
	These varieties are contained in a larger class of varieties called ``Generalised Asanuma varieties" which are well studied in \cite{adv2} and \cite{GG}.
	In \cite{adv}, Gupta has shown that $B^{[1]}=k^{[m+3]}$ but $B \neq k^{[m+2]}$ whenever $f$ is a {\it non-trivial line}, i.e., $k[Z,T]/(f)=k^{[1]}$ but $k[Z,T] \neq k[f]^{[1]}$.
	Since non-trivial lines exist only over fields of positive characteristic (cf. \cite{AM}, \cite{Suz}, \cite{Se}, \cite{Na}), the rings of the form $B$ produced counterexamples to ZCP over fields of positive characteristic. Such rings were further studied in \cite{adv2} and an infinite family of pairwise non-isomorphic varieties having coordinate rings of type \eqref{A1} are obtained, which are counterexamples to the ZCP in positive characteristic. 
	
	In this paper we study varieties whose coordinate rings are given by the following equation.  
	\begin{equation}\label{A}
		A:=\dfrac{k[X_1,\ldots,X_m,Y,Z,T]}{(\alpha(X_1,\ldots,X_m)Y-F(X_1,\ldots,X_m,Z,T))},
	\end{equation}
	such that $\alpha \notin k$ and either $\deg_Z F \geqslant 1$ or $\deg_T F \geqslant 1$. Let $x_1,\ldots,x_m,y,z,t$ denote the images of $X_1,\ldots,X_m,Y,Z,T$ in $A$.
	When $\alpha(X_1,\ldots,X_m)=a_1(X_1) \cdots a_m(X_m)$ and $F=f(Z,T)$, we denote the corresponding rings as $A_{(a_1,\ldots,a_m)}$, that is
	$$
	A_{(a_1,\ldots,a_m)} := \dfrac{k[X_1,\ldots,X_m,Y,Z,T]}{(a_1(X_1)\cdots a_m(X_m)Y-f(Z,T))}.
	$$
	We show that when $k$ is a field of positive characteristic and $f$ is a non-trivial line in $k[Z,T]$, the family
	$$
	\left\{ A_{(a_1,\ldots,a_m)} \mid a_1,\ldots, a_m \in k^{[1]} \right\}
	$$ 
	contains an infinite subfamily of pairwise non-isomorphic rings which are counterexamples to the ZCP but none of them is isomorphic to the varieties of type \eqref{A1}. Thus we produce a new family of counterexamples to the ZCP over fields of positive characteristic (\thref{zcp2}).    
	
	In the next section we will record some basic facts on exponential maps on $k$-algebras and recall some known results which will be used later in this paper.
	
	\section{Preliminaries}	
	
	We begin this section with the basic concept of exponential maps on $k$-algebras.
	
	\medskip
	\noindent
	{\bf Definition.}
	Let $B$ be a $k$-algebra and $\phi_U: B \rightarrow B[U]$ be a $k$-algebra homomorphism. We say that $\phi = \phi_U$ is an {\it exponential map} on $B$, if $\phi$ satisfies the following two properties:
	\begin{enumerate}[\rm(i)]
		\item $\varepsilon_0\phi_U=id_B$, where $\varepsilon_0: B[U]\rightarrow B$ is the evaluation map at $U = 0$.
		
		\item $\phi_V\phi_U = \phi_{V+U}$; where $\phi_V~: B\rightarrow B[V]$ is extended to a $k$-algebra homomorphism $\phi_V:B[U]\rightarrow B[U,V]$ by defining $\phi_V(U) = U$.	
	\end{enumerate}

	Given an exponential map $\phi$ on a $k$-algebra $B$, the ring of invariants of $\phi$ is a subring of $B$ given by 
	\begin{center}
		$B^{\phi} = \left\{ a\in B\,|\, \phi(a) = a\right\}$.
	\end{center}
	An exponential map $\phi$ on $B$ is said to be non-trivial if $B^{\phi} \ne B$. We denote the set of all exponential maps on $B$ by EXP$(B)$.
	The 
	$Derksen$ $invariant$ of $B$ is a subring of $B$ defined by
	\begin{center}
		$\dk(B) = k[B^{\phi}\mid \phi$ is a  non-trivial exponential map on $B]$
	\end{center}
	and the 
	{\it Makar-Limanov invariant} of $B$ is a subring of $B$ defined by
	\begin{center}
		$\ml(B) = \bigcap\limits_{\phi \in \text{EXP}(B)}B^{\phi}$.
	\end{center} 
	
	We list below some useful properties of exponential maps (cf. \cite[Chapter I]{miya}, \cite{cra}, \cite{inv})  on a $k$-domain.
	
	\begin{lem}\thlabel{lemma : properties}
		Let $\phi$ be a non-trivial exponential map on a $k$-domain $B$. Then the following holds: 	\begin{enumerate}[\rm(i)]			
			\item $B^{\phi}$ is factorially closed in $B$ i.e., for any $a,b\in B\setminus\{0\}$, if $ab\in B^{\phi}$ then $a,b\in B^{\phi}$. Hence, $B^{\phi}$ is algebraically closed in $B$.
			
			\item $\td_k(B^{\phi})$ = $\td_k(B)-1$.
			
			\item Let $S$ be a multiplicative closed subset of $B^{\phi}\setminus \{0\}$. Then $\phi$ extends to a non-trivial exponential map $S^{-1}\phi$ on $S^{-1}B$ defined by $S^{-1}\phi(a/s) = \phi(a)/s$, for all $a\in B$ and $s\in S$. Moreover, the ring of invariants of $S^{-1}\phi$ is $S^{-1}(B^{\phi})$.
			
			\item 
			$\phi$ extends to a non-trivial exponential map ${\phi}\otimes id$ on $B\otimes_k\overline{k}$ such that $(B\otimes_k \overline{k})^{{\phi}\otimes id} = B^{\phi} \otimes_k \overline{k}$.
		\end{enumerate}	
	\end{lem}	
	
	We define below a {\it proper} and {\it admissible $\bZ$-filtration} on a $k$-domain $B$.
	
	\medskip
	\noindent
	{\bf Definition.} 
	A collection $\{ B_n\mid n\in \bZ \}$ of $k$-linear subspaces of $B$ is said to be a {\it proper $\bZ$-filtration} if the following properties hold:
	\begin{enumerate}[\rm(i)]
		\item $B_n\subseteq B_{n+1}$ for every $n\in \bZ$.
		
		\item $B = \cup_{n \in \bZ}B_n$.
		
		\item $\cap_{n\in \bZ} B_{n} = \{ 0 \}$.
		
		\item $(B_n \setminus B_{n-1})(B_m \setminus B_{m-1})\subseteq B_{m+n}\setminus B _{m+n -1}$ for all $m,n\in \bZ$.
	\end{enumerate}

	Any proper $\Z$-filtration $\{B_n\}_{n\in \Z}$ on $B$ determines an  {\it associated $\Z$-graded integral domain}
	\begin{center}
		$\gr(B):= \bigoplus\limits_{n\in \Z} \frac{B_n}{B_{n-1}}$
	\end{center}
	and there exists a natural map $\rho: B \rightarrow \gr(B)$ defined by $\rho(a) = a + B_{n-1}$, if $a\in B_n\setminus B_{n-1}$.
	
	\medskip
	\noindent
	{\bf Definition.}	
	A proper $\Z$-filtration $\{B_n\}_{n\in \Z}$ on an affine $k$-domain $B$ is said to be {\it admissible} if there exists a finite generating set $\Gamma$ of $B$ such that, for any $n\in\Z$ and $a\in B_n$, $a$ can be written as a finite sum of monomials in $k[\Gamma] \cap B_n$. 
	
	\smallskip
	
	We now record a theorem on homogenization of exponential map by H. Derksen, O. Hadas and L. Makar-Limanov  \cite{dom}. The following version can be found in \cite[Theorem 2.6]{cra}.
	
	\begin{thm}\thlabel{dhm}
		Let $B$ be an affine $k$-domain with an admissible proper $\mathbb{Z}$-filtration and $gr(B)$ be the induced associated $\mathbb{Z}$-graded domain. Let $\phi$ be a non-trivial exponential map on $B$, then $\phi$
		induces a non-trivial, homogeneous exponential map $\overline{\phi}$ on $gr(B)$ such that $\rho(B^{\phi}) \subseteq gr(B)^{\overline{\phi}}$, where $\rho: B \rightarrow \gr(B)$ denotes the natural map.
	\end{thm}	
	
	Next we quote some important results from \cite{{genadv}} which will be used in this paper.
	\begin{thm}\thlabel{Admissible1}
		Let $R$ be an affine UFD over a field $k$ and
		$$
		A_R:=\dfrac{R[X,Y,Z,T]}{(\alpha(X)X^dY-F(X,Z,T))}\text{ with }\alpha(0) \neq 0 \text{ and }F(0,Z,T)\neq 0.
		$$ 
		Let $x,y,z,t$ denote the images of $X,Y,Z,T$ in $A_R$ respectively.
		Let $w_R$ be a degree function on $A_R$ such that $w_R(r)=0$ for every $r \in R$, $w_R(x)=-1$, $w_R(y)=d$ and $w_R(z)=w_R(t)=0$. 
		Suppose $\gcd_{R[Z,T]}(\alpha(0), F(0,Z,T))=1$, 
		then for any generating set $\{c_1,\dots,c_n\}$ of the affine $k$-algebra $R$, $w_R$ induces an admissible $\Z$-filtration on $A_R$ with respect to the generating set $\{c_1,\dots,c_n,x,y,z,t\}$ such that
		$$gr(A_R)\cong \dfrac{R[X,Y,Z,T]}{(\alpha(0)X^dY-F(0,Z,T))}.$$
	\end{thm}
	
	We now quote special cases of two results from \cite{genadv} which are going to be used in the subsequent section of this paper. 
	We first record the following result (\cite[Proposition 4.20]{genadv}). 
	
	\begin{prop}\thlabel{lin2}
		Let
		 $$A= \dfrac{k[X_1,\ldots,X_m,Y,Z,T]}{(X_1^{r_1}\cdots X_m^{r_m} \alpha_1(X_1)\cdots \alpha_m(X_m)Y-f(Z,T)-X_1 h(X_1,\ldots,X_m,Z,T))}$$
		 be an affine domain such that $r_i>1$, $\alpha_i(0) \neq 0$ for every $i$, $1 \leqslant i \leqslant m$ and for some $h \in k^{[m+2]}$. 
		Let $w_1$ be the degree function on $A$ given by $w_1(x_1)=-1, w_1(y)= r_1,w_1(x_i)= w_1(z) = w_1(t) =0,\, 2\leqslant i\leqslant m$.
		If $\dk(A)$ contains an element with positive $w_1$-value (in particular, if $\dk(A)=A$), then the following statements hold.
		\begin{enumerate}[\rm(i)]
			\item When $m=1$ or $k$ is an infinite field then there exist a system of coordinates $Z_1,T_1$ of $k[Z,T]$ and $a_0,a_1 \in k^{[1]}$, such that $f(Z,T)=a_0(Z_1)+a_1(Z_1)T_1$. 
			\item When $f$ is a line in $k[Z,T]$, then $k[Z,T]=k[f]^{[1]}$.
		\end{enumerate}
	\end{prop}
	
	The next theorem follows from \cite[Theorem 4.22]{genadv}.
	
	\begin{thm}\thlabel{equiv}
		Let $A$ be an affine domain defined as follows
		$$
		A= \dfrac{k[X_1,\ldots,X_m,Y,Z,T]}{(a_1(X_1)\cdots a_m(X_m)Y-f(Z,T))},
		$$
		such that each $a_i(X_i)$ has at least one separable multiple root in $\overline{k}$. Then the following are equivalent.
		\begin{enumerate}
			\item [\rm (i)] $A=k^{[m+2]}$.
			
			\item[\rm (ii)] $k[Z,T]=k[f]^{[1]}$.
		\end{enumerate}
		
	\end{thm}
	
	\section{Main Theorems}\label{sectioniso}
	We begin with the following lemma. 
	\begin{lem}\thlabel{lem0}
		Let $A$ be the affine domain as in \eqref{A} i.e.,
		\begin{equation*}
			A=\dfrac{k[X_1,\ldots,X_m,Y,Z,T]}{(\alpha(X_1,\ldots,X_m)Y-F(X_1,\ldots,X_m,Z,T))}.
		\end{equation*} 
		Then the following statements hold: 
		\begin{itemize}
			\item [\rm (i)]  $k[x_1,\dots,x_m,z,t] \subseteq \dk(A)$.
			
			\item[\rm(ii)] Suppose for every prime divisor $p$ of $\alpha$ in $k[X_1,\ldots,X_m]$, $p \notin A^{*}$ and $F\notin k[X_1,\dots,X_m]$. Then $\ml(A) \subseteq k[x_1,\ldots,x_m]$. 
		\end{itemize}
	\end{lem}
	\begin{proof}
		(i) At first we define two exponential maps  $\phi_{1}$, $\phi_{2}$ on $A$.
		
		\noindent
		$\phi_{1}: A \rightarrow A[U]$ is defined as follows:
		$$\begin{array}{lll}
			\phi_{1}(x_i) &= &x_i,\,1\leqslant i\leqslant m\\
			\phi_{1}(z)&= &z+ \alpha(x_1,\dots,x_m)U\\
			\phi_{1}(t)&=&t\\
			\phi_{1}(y)&=&\frac{F(x_1,\dots,x_m, z+\alpha U, t)}{\alpha} = y+ Uv(x_1,\dots,x_m,z,t,U), \text{~for some~}v\in k[x_1,\dots,x_m,z,t,U]
		\end{array}$$
		%
		\noindent
		Next $\phi_{2}: A \rightarrow A[U]$ is  defined below:
		$$\begin{array}{lll}
			\phi_2(x_i) &= &x_i,\,1\leqslant i\leqslant m\\
			\phi_{2}(z)&= &z\\
			\phi_{2}(t)&=&t+ \alpha(x_1,\dots,x_m)U\\
			\phi_{2}(y)&=&\frac{F(x_1,\dots,x_m, z, t+ \alpha U)}{\alpha} = y+ Uw(x_1,\dots,x_m,z,t,U), \text{~for some~}w\in k[x_1,\dots,x_m,z,t,U]
		\end{array}$$
		Clearly 
		\begin{equation}\label{phi1}
			k[x_1,\dots,x_m,t] \subseteq A^{\phi_{1}} \subseteq k[x_1,\dots,x_m,\alpha^{-1},t] 
		\end{equation}
		and
		\begin{equation}\label{phi2}
			k[x_1,\dots,x_m,z] \subseteq A^{\phi_{2}}\subseteq k[x_1,\dots,x_m,\alpha^{-1},z] 
		\end{equation}
		Therefore, $k[x_1,\dots,x_m,z,t] \subseteq \dk(A)$.
		
		\smallskip
		\noindent
		(ii)
		We now assume that $p \notin A^*$ for every prime divisor $p$ of $\alpha$ and $F\notin k[X_1,\dots,X_m]$. Then it follows that 
		$\ml(A) \subseteq A^{\phi_{1}} \cap A^{\phi_{2}} \cap A= k[x_1,\dots,x_m]$.
	\end{proof}
	
	The following result describes the isomorphisms between two affine domains of the form $A$ as in \eqref{A} upto certain condition on their Makar-Limanov and Derksen invariants. 
	
	\begin{thm}\thlabel{iso}
		Let 
		$$
		A_1:=\frac{k[X_1,\ldots,X_m,Y,Z,T]}{(\alpha_1(X_1,\ldots,X_m)Y-F_1(X_1,\ldots,X_m,Z,T))}
		$$
		and 
		$$
		A_2:= \frac{k[X_1,\ldots,X_m,Y,Z,T]}{(\alpha_2(X_1,\ldots,X_m)Y-F_2(X_1,\ldots,X_m,Z,T))}.
		$$
		Let $x_1,\ldots,x_m,y,z,t$ and $x_1^{\prime},\ldots,x_m^{\prime},y^{\prime},z^{\prime},t^{\prime}$ denote the images of $X_1,\ldots,X_m,Y,Z,T$ in $A_1$ and $A_2$ respectively.
		Suppose that $B_1:=\dk(A_1)=k[x_1,\ldots,x_m,z,t]$, $E_1:=\ml(A_1)=k[x_1,\ldots,x_m]$, $B_2:=\dk(A_2)=k[x_1^{\prime},\ldots,x_m^{\prime},z^{\prime},t^{\prime}]$ and $E_2:=\ml(A_2)=k[x_1^{\prime},\ldots,x_m^{\prime}]$.
		Let $\alpha_1=\prod_{i=1}^{n} p_i^{r_i}$ and $\alpha_2= \prod_{j=1}^{n^{\prime}} q_j^{s_j}$ be the prime factorisations of $\alpha_1$ and $\alpha_2$ in $k[X_1,\ldots,X_m]$ respectively. 
		Suppose $\phi: A_1 \rightarrow A_2$ is an isomorphism, then the following hold:
		\begin{itemize}
			\item[\rm (i)] $\phi$ restricts to an isomorphism from $B_1$ to $B_2$ and from $E_1$ to $E_2$.
			
			\item [\rm (ii)] 
			For every $j$, $1 \leqslant j \leqslant n^{\prime}$, there exists $i$, $1 \leqslant i \leqslant n$ such that $\phi(p_i) =\lambda_{ij} q_j$ for some $\lambda_{ij} \in k^*$ and $n=n^{\prime}$.

			\item[\rm (iii)]   
			If $\phi(p_i)=\lambda_{ij} q_j$, then $r_i=s_j$.
			
			\item[\rm (iv)] $\phi(\alpha_1)=\gamma \alpha_2$ for some $\gamma \in k^*$.
			
			\item[\rm (v)] $\phi((\alpha_1, F_1)B_1)=(\alpha_2,F_2)B_2$.
			
		\end{itemize} 
	\end{thm}
	\begin{proof}
		(i)  Since $B_i= \dk(A_i)$ and $E_i=\ml(A_i)$ for $i=1,2$, the assertion follows.
		
		\smallskip
		\noindent
		(ii) Since $\phi:A_1 \rightarrow A_2$ is an isomorphism and (i) holds, identifying $\phi(A_1)$ to $A_1$ we can assume that $A_1=A_2=A$, $B_1=B_2=B$ and $E_1=E_2=E$. Now note the following
		\begin{equation}\label{a}
			B \hookrightarrow A \hookrightarrow B\left[\frac{1}{p_1},\ldots, \frac{1}{p_n}\right]
		\end{equation}
		and
		\begin{equation}\label{b}
			B \hookrightarrow A \hookrightarrow B\left[\frac{1}{q_1},\ldots, \frac{1}{q_{n^{\prime}}}\right].
		\end{equation}
		Since $y^{\prime}=\dfrac{F_2(x_1^{\prime},\ldots,x_m^{\prime},z^{\prime},t^{\prime})}{q_1^{s_1}\cdots q_{n^{\prime}}^{s_{n^{\prime}}}} \in A \setminus B$ and \eqref{a} holds, there exists $l >0$ such that 
		$$
		(p_1 \cdots p_n)^l y^{\prime}= \dfrac{(p_1 \cdots p_n)^l F_2(x_1^{\prime},\ldots,x_m^{\prime},z^{\prime},t^{\prime})}{q_1^{s_1}\cdots q_{n^{\prime}}^{s_{n^{\prime}}}} \in B.
		$$
		Since $A$ is an integral domain, for every $j$, $q_j \nmid F_2(x_1^{\prime},\ldots,x_m^{\prime},z^{\prime},t^{\prime})$ in $B$, and hence there exists $i$, $1 \leqslant i \leqslant n$ and $\lambda_{ij} \in k^*$ such that 
		\begin{equation}\label{2}
			p_i =\lambda_{ij} q_j.
		\end{equation}
		Thus $n^{\prime} \leqslant n$.
		Further using the fact that $y \in A \setminus B$, by \eqref{b} we have $n \leqslant n^{\prime}$.
		Therefore the assertion follows.
		
		\smallskip
		\noindent
		(iii) By the given condition we can assume that \eqref{2} holds. We now show that $r_i=s_j$. If possible suppose $r_i<s_j$. 
		Consider the ideal $I= q_j^{s_j}A \cap B=(q_j^{s_j}, F_2(x_1^{\prime},\ldots,x_m^{\prime},z^{\prime},t^{\prime}))B$. Further using \eqref{2}, we have $I=p_i^{s_j}A \cap B=(p_i^{s_j}, p_i^{s_j-r_i} F_1(x_1,\ldots,x_m,z,t))B \subseteq p_iB=q_jB$. But it contradicts the fact that $q_j \nmid F_2(x_1^{\prime},\ldots,x_m^{\prime},z^{\prime},t^{\prime})$ in $B$. Thus $r_i \geqslant s_j$.  
		
		Next if $r_i>s_j$, considering the ideal $J=p_i^{r_i} A \cap B$ and by similar arguments as above, we get that  $F_1(x_1,\ldots,x_m,z,t) \in p_iB$, which is a contradiction.  
		Thus $r_i=s_j$.
		
		\smallskip
		\noindent
		(iv) By \eqref{2} and (iii) 
		we have $p_i^{r_i}= \mu_{ij} q_j^{s_j}$ for some $\mu_{ij} \in k^*$ and $r_i=s_j$. 
		Now using the factorisations of $\alpha_1(x_1,\ldots,x_m)$ and $\alpha_2(x_1^{\prime},\ldots,x_m^{\prime})$ in $B$ we have 
		\begin{equation}\label{3}
			\alpha_1(x_1,\ldots,x_m)= \gamma \alpha_2(x_1^{\prime},\ldots,x_m^{\prime}),
		\end{equation}
		for some $\gamma \in k^*$.
		Thus the assertion follows.
		
		\smallskip
		\noindent
		(v) By \eqref{3}, we have $\alpha_1(x_1,\ldots,x_m)A \cap B=\alpha_2(x_1^{\prime},\ldots,x_m^{\prime})A \cap B$. Thus
		\begin{equation}
			(\alpha_1(x_1,\ldots,x_m), F_1(x_1,\ldots,x_m,z,t))B=(\alpha_2(x_1^{\prime},\ldots,x_m^{\prime}), F_2(x_1^{\prime},\ldots,x_m^{\prime},z^{\prime},t^{\prime}))B
		\end{equation}
		and hence the result follows.
	\end{proof}
	From now on wards for the rest of this section, we assume that $A$ be as in \eqref{A}, with $\alpha(X_1,\ldots,X_m)=a_1(X_1) \cdots a_m(X_m)$ and $F=f(Z,T)+(\prod_{j=1}^{n} p_j )h(X_1,\ldots,X_m,Z,T)$ where $p_j$'s are all prime factors of $\alpha(X_1,\ldots,X_m)$ in $k[\X]$. That means
	\begin{equation}\label{A2}
		A=\dfrac{k[X_1,\ldots,X_m,Y,Z,T]}{\left(a_1(X_1) \cdots a_m(X_m)Y- f(Z,T)-(\prod_{j=1}^{n} p_j )h(X_1,\ldots,X_m,Z,T)\right)}.
	\end{equation}
	
	\begin{rem}\thlabel{r5}
		\rm (i) Let $\lambda_i$ be a root of $a_i(X_i)$ in $\overline{k}$ with multiplicity $r_i\geqslant 1$, for some $i,\mi$. Now note that 
		$$
		A \hookrightarrow \overline{A}:=A \otimes_k \overline{k}:=\dfrac{\overline{k}[X_1,\ldots,X_m,Y,Z,T]}{((X_i-\lambda_i)^{r_i} \alpha^{\prime}(X_1,\ldots,X_m)Y-F(X_1,\ldots,X_m,Z,T))},
		$$
		where $\alpha^{\prime}(X_1,\ldots,\lambda_i,\ldots,X_m) \neq 0$. 
		Consider the degree function $\omega_{\lambda_i}$ on $\overline{A}$ given by, 
		$$\omega_{\lambda_i}(x_i-\lambda_i)=-1,\,
		\omega_{\lambda_i}(y)=r_i,\, \omega_{\lambda_i}(x_j)=0 \text{ for } 1\leqslant j \leqslant m,\, j \neq i,\, \omega_{\lambda_i}(z)=0,\, \omega_{\lambda_i}(t)=0.$$ By \thref{Admissible1}, $\omega_{\lambda_i}$ induces an admissible $\bZ$-filtration on $\overline{A}$ with respect to the generating set $\{x_1,\ldots, x_i-\lambda_i, \ldots, x_m, y, z, t\}$. 
		
		\smallskip
		\noindent
		(ii) 
		Since the filtration induced by $\omega_{\lambda_i}$ is admissible with respect to $\{x_1,\ldots, x_i-\lambda_i, \ldots,x_m,y,z,t\}$, if $b\in A$ be such that $\omega_{\lambda_i}(b) \leqslant 0$, then $b \in \overline{k}[x_1,\ldots,x_{m}, (x_i-\lambda_i)^{r_i}y, z, t]$.

		\smallskip
		\noindent
		(iii)
		Let $b \in A$ be such that $\omega_{\lambda_i}(b)>0$, then the highest degree homogeneous summand of $b$ with respect to $\omega_{\lambda_i}$ is divisible by $y$ in $\overline{A}$.  
	\end{rem}

	The following result describes $\ml(A)$ when $\dk(A)=k[x_1,\ldots,x_m,z,t]$.
	\begin{prop}\thlabel{ml}
		Let $\phi$ be a non-trivial exponential map on $A$ such that $A^{\phi} \subseteq k[x_1,\ldots,x_m,z,t]$. Then $k[x_1,\ldots,x_m] \subseteq A^{\phi}$. In particular,  
		if $\dk(A)=k[x_1,\ldots,x_m,z,t]$, then $\ml(A)=k[x_1,\ldots,x_m]$.
	\end{prop}
	
	\begin{proof}
		If possible suppose $x_i \notin A^{\phi}$, for some $i,\mi$. Consider the ring $\overline{A}:=A \otimes_k \overline{k}$. From \eqref{A2}, it is clear that 
		\begin{equation}
			\overline{A}=\dfrac{\overline{k}[X_1,\ldots,X_m,Y,Z,T]}{(X_1^{r_1}\cdots X_m^{r_m}\alpha_1(X_1)\cdots \alpha_m(X_m)Y-f(Z,T)-\overline{h}(X_1,\ldots,X_m,Z,T) )},
		\end{equation}
		for some $r_j \geqslant 1$, $\overline{h}\in \overline{k}^{[m+2]}$, $\alpha_j\in \overline{k}^{[1]}$ with $\alpha_j(0) \neq 0$ for every $j$, $1 \leqslant l \leqslant m$ and $X_i\mid \overline{h}$.
		Let $\overline{x_1},\ldots, \overline{x_m}, \overline{y}, \overline{z}, \overline{t}$ denote the images of $X_1,\ldots,X_m,Y,Z,T$ in $\overline{A}$.
		Consider the non-trivial exponential map $\overline{\phi}$ on $\overline{A}$ induced by $\phi$ (cf. \thref{lemma : properties}(iv)).
		Since $x_i \notin A^{\phi}$, it follows that $\overline{x_i} \notin \overline{A}^{\overline{\phi}}$.
		Further, since $A^{\phi} \subseteq k[x_1,\ldots,x_m,z,t]$, we have $\overline{A}^{\overline{\phi}} \subseteq \overline{k}[\overline{x_1},\ldots,\overline{x_m}, \overline{z}, \overline{t}]$. 
		By \thref{lemma : properties}(ii), $\td_k(\overline{A}^{\overline \phi})$ = $m+1$.
		Let $\{f_1, \ldots,f_{m+1}\}$ be an algebraically independent set of elements in $\overline{A}^{\overline{\phi}}$. Suppose 
		$$
		f_j=g_j(\overline{x_1},\ldots, \overline{x_{i-1}},\overline{x_{i+1}},\ldots,\overline{x_m},\overline{z}, \overline{t}) + \overline{x_i} h_j(\overline{x_1},\ldots,\overline{x_m},\overline{z}, \overline{t}),
		$$   
		for every $j$, $1 \leqslant j \leqslant m+1$. 
		If there exists a polynomial $P$ such that $P(g_1,\ldots,g_{m+1})=0$ then $\overline{x_i} \mid P(f_1,\ldots,f_{m+1})$. Now this would contradict the fact that $\overline{x_i} \notin \overline{A}^{\overline{\phi}}$ as $P(f_1,\ldots,f_{m+1}) \in \overline{A}^{\overline{\phi}}$ and $\overline{A}^{\overline{\phi}}$ is factorially closed (cf. \thref{lemma : properties}(i)).	
		Therefore, $\{g_1,\ldots,g_{m+1}\}$ is an algebraically independent set of elements in $\overline{k}[\overline{x_1},\ldots, \overline{x_{i-1}},\overline{x_{i+1}},\ldots,\overline{x_m},\overline{z}, \overline{t}]$.
		We now consider the degree function $\omega_i$ on $\overline{A}$, given by $\omega_i(\overline{x_i})=-1$,
		$\omega_i(\overline{y})=r_i$, $\omega_i(\overline{x_j})=0$, for $1 \leqslant j \leqslant m$, $j \neq i$, $\omega_i(\overline{z})=\omega_i(\overline{t})=0$. Then by \thref{Admissible1}, $\omega_i$ induces an admissible $\bZ$-filtration on $\overline{A}$ and the associated graded ring  is 
		\begin{equation}\label{6}
			\tilde{A}= \dfrac{\overline{k}[X_1,\ldots,X_m,Y,Z,T]}{(X_i^{r_i}\tilde{\alpha}(X_1,\ldots,X_{i-1},X_{i+1},\ldots,X_m)Y-f(Z,T) )},
		\end{equation}
		for some $\tilde{\alpha} \in \overline{k}[X_1,\ldots,X_{i-1},X_{i+1},\ldots,X_m]$.
		For any element $b \in \overline{A}$, let $\tilde{b}$ denotes its image in $\tilde{A}$.
		Now by \thref{dhm}, $\overline{\phi}$ induces a non-trivial exponential map $\tilde{\phi}$ on $\tilde{A}$ such that $\tilde{f_j} \in \tilde{A}^{\tilde{\phi}}$ for every $j$, $1 \leqslant j \leqslant m+1$. 
		Note that $$\tilde{f_j}=g_j(\widetilde{x_1},\ldots, \widetilde{x_{i-1}},\widetilde{x_{i+1}},\ldots,\widetilde{x_m},\widetilde{z},\widetilde{t}), \text{ for every } j, 1\leqslant j\leqslant m.$$ 
		Since $g_1,\ldots,g_{m+1}$ are algebraically independent it follows that $\overline{k}[\widetilde{x_1},\ldots, \widetilde{x_{i-1}},\widetilde{x_{i+1}},\ldots,\widetilde{x_m},\widetilde{z},\widetilde{t}]$ is algebraic over $\overline{k}[g_1,\ldots,g_{m+1}]$.
		Hence $\overline{k}[\widetilde{x_1},\ldots, \widetilde{x_{i-1}},\widetilde{x_{i+1}},\ldots,\widetilde{x_m},\widetilde{z},\widetilde{t}] \subseteq \tilde{A}^{\tilde{\phi}}$. But by \eqref{6}, this contradicts that $\tilde{\phi}$ is non-trivial.
		Thus we have $x_i \in A^{\phi}$ and since $i$ is arbitrary, $k[x_1,\ldots,x_m] \subseteq A^{\phi}$. 
		
		If $\dk(A)=k[x_1,\ldots,x_m,z,t]$, then for every prime divisor $p$ of $a$, $p \notin A^*$ and $f\notin k$. Therefore, by \thref{lem0}, we have $\ml(A)=k[x_1,\ldots,x_m]$. 
	\end{proof}
	
	Now we record an easy lemma which will be useful in proving the next two results which will describe $\dk(A)$, upto a certain condition on the roots of $a_i(X_i)$ for every $i$, $1 \leqslant i \leqslant m$. 
	
	\begin{lem}\thlabel{line}
		Let $f(Z,T) \in k[Z,T]$ be such that $k[Z,T]/(f)=k^{[1]}$ and there exists a system of coordinates $\{Z_1,T_1\}$ in $k[Z,T]$ such that $f(Z,T)=b_0(Z_1)+b_1(Z_1)T_1$ for some $b_0, b_1 \in k^{[1]}$. Then $k[Z,T]=k[f]^{[1]}$.  
	In particular, when $k$ is a field of positive characteristic and $g(Z,T)$ is a non-trivial line, then $g(Z,T)$ is not linear with respect to any system of coordinate in ${k}[Z,T]$.
	\end{lem}
	
	\begin{lem}\thlabel{r4}
		Suppose for every $i$, $1 \leqslant i \leqslant m$ there exists a multiple root $\lambda_{i}$ of $a_i(X_i)$ in $\overline{k}$, such that $f(Z,T)$ is not linear with respect to any system of coordinate in $\overline{k}[Z,T]$. Then the following hold.
		
		\smallskip
		\noindent
		{\rm (i)}
		Let $w_{\lambda_{i}}$ denotes the degree function on $\overline{A}:= A \otimes_k \overline{k}$ as in \thref{r5}\rm(i), then for every element $a \in \dk(A)$, $w_{\lambda_{i}}(a) \leqslant 0$. In particular, $\dk(A) \subsetneq A$.
		
		\smallskip
		\noindent
		{\rm (ii)} If $\mu$ denotes the multiplicity of the root $\lambda_{i}$ in $a_i(X_i)$, then $\dk(A) \subseteq \overline{k}[x_1,\ldots,x_m, (x_i-\lambda_{i})^{\mu}y,z,t] $. 
	\end{lem}
	\begin{proof}
		(i) Follows from \thref{lin2}.
		
		\smallskip
		\noindent
		(ii)  Follows from (i) and \thref{r5}(ii).
	\end{proof}
	
	\begin{cor}\thlabel{c5}
		For every $i$, let  $a_i(X_i)= \prod_{j=1}^{n_i} (X_i- \lambda_{ij})^{\mu_{ij}}$ in $\overline{k}[X_i]$, with $\mu_{ij} >1$.
		Suppose that $f(Z,T)$ is not linear with respect to any coordinate system in $\overline{k}[Z,T]$.
		Then $\dk(A)=k[x_1,\ldots,x_m,z,t]$.
	\end{cor}
	\begin{proof}
		By \thref{r4}(ii), $\dk(A) \subseteq \bigcap_{i=1}^m (\bigcap_{j=1}^{n_i} \overline{k}[x_1,\ldots,x_m, (x_i-\lambda_{ij})^{\mu_{ij}}y,z,t]) =\overline{k}[x_1,\ldots,x_m,z,t]$. Therefore, $\dk(A) \subseteq k[x_,\ldots,x_m,z,t]$. Now the assertion follows by \thref{lem0}(i).
	\end{proof}
	
	The following result describes isomorphisms between two rings of the form $A$ as in \eqref{A2}, when every root of $a_i(X_i)$ is a multiple root. 
	
	\begin{thm}\thlabel{iso2}
		Let 
		$$
		A_1:=\frac{k[X_1,\ldots,X_m,Y,Z,T]}{(a_1(X_1)\ldots a_m(X_m)Y-f(Z,T)-h_1(X_1,\ldots,X_m,Z,T))}
		$$
		and 
		$$
		A_2:= \frac{k[X_1,\ldots,X_m,Y,Z,T]}{(b_1(X_1)\ldots b_m(X_m)Y-g(Z,T)-h_2(X_1,\ldots,X_m,Z,T))}
		$$
		be such that $a_i(X_i):=\prod_{1 \leqslant j \leqslant n}(X_i-\lambda_{ij})^{d_{ij}}$ and $b_i(X_i):=\prod_{1 \leqslant j \leqslant n^{\prime}}(X_i-\mu_{ij})^{e_{ij}}$ in $\overline{k}[X_i]$ with $d_{ij},e_{ij} >1$ for all $i,j$
		and $\prod_{1 \leqslant j \leqslant n}(X_i-\lambda_{ij})\mid h_1$,  $\prod_{1 \leqslant j \leqslant n}(X_i-\mu_{ij})\mid h_2$ in $\overline{k}[\X]$. 
		Let $x_1,\ldots,x_m,y,z,t$ and $x_1^{\prime},\ldots,x_m^{\prime},y^{\prime},z^{\prime},t^{\prime}$ denote the images of $X_1,\ldots,X_m,Y,Z,T$ in $A_1$ and $A_2$ respectively.
		Let $E_1:=k[x_1,\ldots,x_m]$, $E_2:=k[x_1^{\prime},\ldots,x_m^{\prime}]$, $B_1:=k[x_1,\ldots,x_m,z,t]$ and $B_2:=k[x_1^{\prime},\ldots,x_m^{\prime},z^{\prime},t^{\prime}]$.
		Suppose $f(Z,T)$ and $g(Z,T)$ are not linear with respect to any system of coordinate in $\overline{k}[Z,T]$ and $\phi: A_1 \rightarrow A_2$ is an isomorphism. Then the following hold.
		\begin{itemize}
			\item[\rm (i)] $\phi$ restricts to an isomorphism from $B_1$ to $B_2$ and from $E_1$ to $E_2$.
			
			\item [\rm (ii)] For every $l$, $1 \leqslant l \leqslant m$, there exists $i$, $1 \leqslant i \leqslant m$ such that $\phi(x_i)= \nu x_l^{\prime} + \mu$, for some $\nu \in k^{*}$ and $\mu \in k$.
			
			\item[\rm (iii)]  If  $\phi(x_i)= \nu x_l^{\prime} + \mu$, then $a_i(X_i)$ and $b_l(X_l)$ have equal number of roots in $\overline{k}$ with equal multiplicities. Further $\phi(a_i(x_i))=\gamma b_l(x_l^{\prime})$ for some $\gamma \in k^*$.
		\end{itemize} 
	\end{thm}
	\begin{proof}
		(i) 
		Since $f(Z,T)$ and $g(Z,T)$ are both not linear with respect to any system of coordinate in $\overline{k}[Z,T]$, by \thref{c5}, $\dk(A_i)=B_i$ and hence by \thref{ml}, $\ml(A_i)=E_i$ for $i=1,2$. Thus the result follows.

		\smallskip
		\noindent
		(ii)
		Since $\phi:A_1 \rightarrow A_2$ is an isomorphism, it induces an isomorphism $\overline{\phi}: \overline{A_1}:= A_1 \otimes_k \overline{k} \rightarrow \overline{A_2}:=A_2 \otimes_k \overline{k}$.	
		Further from \thref{c5} it follows that $\dk(\overline{A_i})=\overline{B_i}:=B_i \otimes_k \overline{k}$ and $\ml(\overline{A_i})=\overline{E_i}:=E_i \otimes_k \overline{k}$. 
		Now since $\overline{\phi}$ is an isomorphism, from \thref{iso}(ii) we know that for every $l$, $1 \leqslant l \leqslant m$ and a prime divisor $(x_l^{\prime} - \mu_{lj_l})$ of $b_l(x_l^{\prime})$ in $\overline{B_2}$, there exist some $i,\mi$ and a prime divisor $(x_i-\lambda_{ij_i})$ of $a_i(x_i)$ in $\overline{B_1}$
		such that $\overline{\phi}(x_i-\lambda_{ij_i})=\nu (x_l^{\prime} - \mu_{lj_l})$ for some $\nu \in \overline{k}^*$. Now since $\overline{\phi} \mid_{B_1}= \phi$ and $\phi(B_1)=B_2$, we have the desired result.
		
		\smallskip
		\noindent
		(iii) If $\phi(x_i)= \nu x_l^{\prime} + \mu$, 
		considering $\overline{\phi}: \overline{A_1} \rightarrow \overline{A_2}$ we get that a prime divisor $(x_i-\lambda_i)$ of $a_i(x_i)$ is mapped to a prime divisor $(x_l^{\prime}-\mu_l)$ of $b_l(x_l^{\prime})$.
		Now from \thref{iso}(ii), it follows that there is an one to one correspondence between the roots of $a_i(X_i)$ and $b_l(X_l)$ in $\overline{k}$. Now for a root $\lambda_{ij_i}$ of $a_i(X_i)$ if 
		$\overline{\phi}(x_i-\lambda_{ij_i})=(x_l^{\prime} - \mu_{lj_l})$
		for some root $\mu_{lj_l}$ of $b_l(X_l)$, then by \thref{iso}(iii), $\lambda_{ij_i}$ and $\mu_{lj_l}$ have the same multiplicity.
		Therefore, it follows that $\phi(a_i(x_i))= \gamma b_l(x_l^{\prime})$ for some $\gamma \in k^*$.
	\end{proof}
	
	The next result characterizes the automorphisms of $A$ as in \eqref{A2} when every root of $a_i(X_i)$ is a multiple root for every $i$, $1 \leqslant i \leqslant m$.
	
	\begin{thm}\thlabel{auto}
		Let $A$ be the affine domain as in \eqref{A2}, with  $a_i(X_i)= \prod_{1 \leqslant j \leqslant n_i} (X_i- \lambda_{ij})^{d_{ij}}$ in $\overline{k}[X_i]$ with $d_{ij}>1,\, \mi$.
		Let $x_1,\ldots,x_m,y,z,t$ be the images of $X_1,\ldots,X_m,Y,Z,T$ in $A$.
		Suppose that $f(Z,T)$ is not linear with respect to any coordinate system in $\overline{k}[Z,T]$. If $\phi \in \Aut_{k}(A)$, then the following hold:
		
		\begin{enumerate}
			\item [\rm(i)]  $\phi$ restricts to an automorphism of $E:=k[x_1,\ldots,x_m]$ and $B:=k[x_1,\ldots,x_m,z,t]$.
			
			\item[\rm(ii)] For every $i$, $1 \leqslant i \leqslant m$, there exists $j$, $1 \leqslant j \leqslant m$ such that $\phi(a_i(x_i))= \gamma a_j(x_j)$ for some $\gamma \in k^*$.
			
			\item[\rm(iii)] $\phi(I)=I$, where $I=\left( \alpha(x_1,\ldots,x_m), F(x_1,\ldots,x_m,z,t) \right)B$.

		\end{enumerate}
		
		Conversely, if $\phi \in \End_k(A)$ satisfies conditions \rm(i) and \rm(iii), then $\phi \in \Aut_k(A)$.
	\end{thm}
	\begin{proof}
		(i), (ii) and (iii) follows from \thref{iso2}(i), \thref{iso2}(iii) and \thref{iso}(v) respectively. 
		
		We now show the converse part. 
		From (i) and (iii) it follows that $\phi(B)=B$, $\phi(E)=E$ and $\phi(I)=I$. Hence $\phi(I \cap E)=I \cap E$.
		Since 
		$
		I \cap E=\left( \alpha(x_1,\ldots,x_m) \right)E,
		$ 
		\begin{equation}\label{7}
			\phi(\alpha(x_1,\ldots,x_m)) = \gamma \alpha(x_1,\ldots,x_m),
		\end{equation}
		for some $\gamma \in k^*$. 
		Now since $\phi$ is an automorphism of $B$ and $A \subseteq B[\alpha(x_1,\ldots,x_m)^{-1}]$, using \eqref{7} it follows that $\phi$ is an injective endomorphism of $A$.
		Therefore, it is enough to show that $\phi$ is surjective. For this, it is enough to find a preimage of $y$ in $A$. Since $\phi(I)=I$, we have 
		\begin{equation}\label{8}
			F= \alpha(x_1,\ldots,x_m) u(x_1,\ldots,x_m,z,t) + \phi(F) v(x_1,\ldots,x_m,z,t), 
		\end{equation}
		for some $u,v \in B$. Since
		$
		y= \frac{F(x_1,\ldots,x_m,z,t)}{\alpha(x_1,\ldots,x_m)},
		$ using \eqref{7} and \eqref{8} we have
		\begin{equation}\label{y}
			y=u(x_1,\ldots,x_m,z,t) + \frac{\phi(F) v(x_1,\ldots,x_m,z,t)}{\gamma^{-1} \phi(\alpha(x_1,\ldots,x_m))}.
		\end{equation}
		Since $\phi(B)=B$, there exist $\widetilde{u},\widetilde{v} \in B$ such that $\phi(\widetilde{u})=u$ and $\phi(\widetilde{v})=v$. And hence from \eqref{y}, we get that 
		$y= \phi(\widetilde{u}+\gamma y\widetilde{v})$, where $\widetilde{u}+\gamma y\widetilde{v} \in A$.
	\end{proof}
	
	We now recall the following result from \cite[Proposition 3.6]{adv}.
	\begin{thm}\thlabel{stiso}
		Let $R$ be an integral domain, $\pi_1,\ldots,\pi_n \in R$ and $\pi=\pi_1 \cdots \pi_n$. Let $G(Z,T) \in R[Z,T]$ be such that
		$R[Z,T]/(\pi, G(Z,T)) \cong_R (R/\pi)^{[1]}$. For positive integers $r_1,\ldots,r_n$, set
		$$
		D:=\dfrac{R[Y,Z,T]}{(\pi_1^{r_1} \cdots \pi_n^{r_n}Y-G(Z,T))}.
		$$
		Then $D^{[1]}=R^{[3]}.$
	\end{thm}
	Let
	$$
	A(r_1,\ldots,r_n,f):= \dfrac{k[X_1,\ldots,X_m,Y,Z,T]}{(X_1^{r_1}\cdots X_m^{r_m}Y-f(Z,T))},
	$$
	where $r_i>1$ for $i$, $1 \leqslant i \leqslant m$ and $f$ is a non-trivial line in $k[Z,T]$.
	Let 
	$$
	\Omega_1:= \left\{A(r_1,\ldots,r_m,f) \mid r_i>1 \text{~and~} f \text{~is a non-trivial line in~} k[Z,T] \right\}.
	$$ 
	In \cite{adv} Gupta has shown the rings $A(r_1,\ldots,r_m,f)$ are counter example to the Zariski Cancellation Problem in positive characteristic. Whereas in \cite[corollary 4.4]{adv2}, it has been shown that $\Omega_1$ contains an infinite family of pairwise non-isomorphic rings which are counter example to the Zariski Cancellation Problem.
	\begin{cor}\thlabel{zcp2}
		Let $k$ be a field of positive characteristic. Then there exists an infinite family of pairwise non-isomorphic varieties which are counterexamples to the Zariski Cancellation Problem in higher dimensions( $\geqslant 3$), and are also non-isomorphic to the members of $\Omega_1$. 
	\end{cor}
	\begin{proof}
		Let $f(Z,T)$ be a non-trivial line in $k[Z,T]$.
		Consider the following family of varieties:
		\begin{gather*}
			\Omega_2:= \{ A_{(a_1,\ldots,a_m)}:=\frac{k[X_1,\ldots,X_m,Y,Z,T]}{(a_1(X_1)\cdots a_m(X_m)Y-f(Z,T))} \mid \text{~each~} a_i \text{~has only multiple roots in~} \overline{k}\\ \text{~and some~} a_i \text{~has at least two distinct roots~} \}. 
		\end{gather*}
		By \thref{stiso}, $A_{(a_1,\ldots,a_m)}^{[1]}=k^{[m+3]}$, for every $A_{(a_1,\ldots,a_m)} \in \Omega$. Now by \thref{equiv}, $A_{(a_1,\ldots,a_m)} \otimes_k \overline{k} \neq \overline{k}^{[m+2]}$. Hence $A_{(a_1,\ldots,a_m)}$ is counterexample to the Zariski Cancellation Problem.
		
		Now since $f(Z,T)$ be a non-trivial line in $k[Z,T]$, by \thref{line}, $f(Z,T)$ is not linear with respect to any coordinate system in $\overline{k}[Z,T]$.
		Hence by \thref{iso2}(ii) and (iii), $A_{(a_1,\ldots,a_m)}\in \Omega_2$ is not isomorphic to the rings in $\Omega_1$.
		Now we consider two sets of polynomials $S_1=\{a_1(X_1),\ldots,a_m(X_m)\}$ and $S_2=\{b_1(X_1),\ldots,b_m(X_m)\}$ such that each $a_i(X_i)$, $b_i(X_i)$ has only multiple roots and total number of roots of $a_1\dots a_m$ is different from total number of roots of $b_1\dots b_m$. 
		Then  $A_{(a_1,\ldots,a_m)} \ncong A_{(b_1,\ldots,b_m)}$, by \thref{iso2}(iii). Therefore, $\Omega_2$ contains infinitely many pairwise non-isomorphic varieties which are counterexamples to the ZCP in dimension $m+2$ for $m \geqslant 1$.
	\end{proof}


\begin{thebibliography}{9999}
		\bibitem{aeh} S.S. Abhyankar, P. Eakin and W. Heinzer, {\it On the uniqueness of the coefficient
			ring in a polynomial ring}, J. Algebra {\bf 23} (1972), 310--342.
		
		\bibitem{AM} S. S. Abhyankar  and T. T. Moh, {\it Embeddings of the line in the plane}, 
		J. Reine Angew. Math. {\bf 276} (1975), 148--166.
		
		\bibitem{cra} A. J. Crachiola, 
		{\it The hypersurface $x + x^2y + z^2 + t^3 = 0$ over a field of arbitrary characteristic},
		Proc. Amer. Math. Soc. {\bf{134}} (2005), 1289--1298.
		
		\bibitem{dom} H. Derksen, O. Hadas and L. Makar-Limanov, 
		{\it Newton polytopes of invariants of additive group actions}, 
		J. Pure Appl. Algebra {\bf{156}} (2001), 187--197.
		
		\bibitem{fuj} T. Fujita, {\it On Zariski problem}, Proc. Japan Acad. {\bf 55A} (1979) 106--110.
		
		\bibitem{adv2}  P. Ghosh and N. Gupta, \textit{On the triviality of a family of linear hyperplanes},
		Adv. Math. {\bf {428}} (2023), 109166.
		
		\bibitem{GG} P. Ghosh and N. Gupta, \textit{On Generalised Danielewski and Asanuma varieties},
		J. Algebra {\bf 632} (2023), 226--250.
		
		\bibitem{genadv} P. Ghosh, N. Gupta and A. Pal, {\it 	On the Epimorphism Problem for a family of linear hypersurfaces in $\A_k^{n}$}, preprint.
		
		\bibitem{inv} N. Gupta, 
		{\it On the cancellation problem for the affine space $\mathbb{A}^3$ in characteristic $p$}, 
		Invent. Math. {\bf 195} (2014), 279--288.
		
		\bibitem{adv} N. Gupta, 
		{\it On Zariski's Cancellation Problem in positive characteristic}, 
		Adv. Math. {\bf {264}} (2014), 296--307.
		
		\bibitem{bhgu}
		S.M. Bhatwadekar and Neena Gupta, {\it A note on the cancellation property of
			$k[X,Y]$}, Journal of Algebra and its Applications (special issue in honour of Prof. Shreeram S.
		Abhyankar), {\bf 14(9)} (2015), 15400071–5.
		
		\bibitem{miya} M. Miyanishi, {\it Lectures on Curves on rational and unirational surfaces}, Narosa Publishing House, New Delhi (1978).
		
		\bibitem{miyasug} M. Miyanishi and T. Sugie, {\it Affine surfaces containing cylinderlike open sets},
		J. Math. Kyoto Univ. {\bf 20} (1980) 11--42.
		
		\bibitem{Na} M. Nagata, {\it On automorphism group of $k[X, Y]$},
		Kyoto Univ. Lec. Math.  {\bf 5}, Kinokuniya, Tokyo (1972).
		
		\bibitem{russ} P. Russell, {\it On affine-ruled rational surfaces}, Math. Ann. {\bf 255} (1981) 287--302.
		
		\bibitem{Se} B. Segre, 
		{\it Corrispondenze di M{\"o}bius e trasformazioni cremoniane intere},
		Atti Accad. Sci. Torino. Cl. Sci. Fis. Mat. Nat. {\bf 91} (1956/1957), 3--19.
		
		\bibitem{Suz} M. Suzuki, 
		{\it Propri\'{e}t\'{e}s topologiques des polyn\^{o}mes de deux 
			variables complexes, et automorphismes alg\'{e}briques de l'espace $\bC^2$}, 
		J. Math. Soc. Japan {\bf 26} (1974), 241--257. 
	\end{thebibliography}
\end{document}